\documentclass{amsart}

\usepackage{esint}
\usepackage{amsfonts}
\usepackage{amssymb}
\usepackage{xcolor}
\usepackage{url, hyperref}

\newtheorem{theorem}{Theorem}[section]

\newtheorem{corollary}[theorem]{Corollary}

\newtheorem{proposition}[theorem]{Proposition}

\numberwithin{equation}{section}

\begin{document}
\title[Bottom spectrum and parabolicity of 3-manifolds]{Bottom spectrum and parabolicity of 3-manifolds with scalar curvature lower bound}

\begin{abstract}
Under a necessary topological assumption, two global results are established for complete three dimensional manifolds. The first one
provides a sharp upper bound for the bottom spectrum in terms of the scalar curvature lower bound. The second one shows that such manifolds 
do not admit any positive Green's function if the scalar curvature is bounded from below by a positive constant.
\end{abstract}

\author{Ovidiu Munteanu and Jiaping Wang}

\address{Department of Mathematics, University of Connecticut, Storrs, CT
06268, USA}
\email{ovidiu.munteanu@uconn.edu}
\address{School of Mathematics, University of Minnesota, Minneapolis, MN
55455, USA}
\email{jiaping@math.umn.edu}

\maketitle

\section{Introduction}

In this short note, we aim to study two issues for 
complete three dimensional manifolds with scalar curvature
lower bound. One concerns the size of the 
bottom spectrum and another the nonexistence of positive Green's functions.

Recall that the bottom spectrum $\lambda _{1}\left( M\right) $ of a complete manifold $M$
is defined to be the smallest 
$\lambda \in \sigma(\Delta),$ where $\sigma(\Delta)$ denotes the spectrum of the Laplacian of $M.$
Alternatively, it may be characterized as the optimal constant in the  Poincar\'e inequality
\begin{equation*}
\lambda _{1}\left( M\right)\,\int_{M}\varphi ^{2}\leq \int_{M}\left\vert \nabla \varphi \right\vert ^{2}
\end{equation*}
for all smooth function $\varphi$ with compact support on $M.$

According to Cheng \cite{C}, the bottom spectrum $\lambda_1(M)$ of an $n$-dimensional complete manifold $M$ with its
Ricci curvature $\mathrm{Ric}\geq -\left( n-1\right) K$ for some nonnegative
constant $K$ must satisfy the following sharp upper bound:
\begin{equation*}
\lambda _{1}\left( M\right) \leq \frac{\left( n-1\right) ^{2}}{4}K.
\end{equation*}

This conclusion fails under a scalar curvature lower bound for $n\geq 4$ as the product manifold $\mathbb{H}^{n-2}\times \mathbb{S}^2(r)$ of the hyperbolic space
$\mathbb{H}^{n-2}$ with the sphere $\mathbb{S}^2(r)$ of radius $r$ has positive bottom spectrum, yet its scalar curvature is positive for $r$ small. 
Even for $n=3,$ the conclusion does not hold, either. Indeed, the universal cover $M$ of the connected sum 
$N=\left(\mathbb{S}^2\times \mathbb{S}^{1}\right)\#\left(\mathbb{S}^2\times \mathbb{S}^{1}\right)$ carries a metric with uniformly positive scalar curvature 
as $N$ does so \cite{GL, SY1}. Note that the first fundamental group of $N$ is a free group on two generators, hence nonamenable.
According to a result of Brooks \cite{B}, the bottom spectrum $\lambda_1(M)$ of $M$ is positive. 

One of our purposes here is to show that, nevertheless, a result parallel to Cheng's still holds after excluding the aforementioned example through
a suitable topological assumption.

\begin{theorem}
\label{A} Let $\left( M,g\right) $ be a three dimensional complete
Riemannian manifold with scalar curvature $S\geq -6K$ on $M$ for some
nonnegative constant $K.$ Suppose that $M$ has finitely many ends and its
first Betti number $b_{1}(M)<\infty.$ Then the bottom spectrum of $M$
satisfies
\begin{equation*}
\lambda _{1}\left( M\right) \leq K.
\end{equation*}
\end{theorem}

Note that the theorem is not applicable to the example mentioned above as it has infinitely many ends. 
It should also be remarked that the same conclusion has been established previously in \cite{MW} under an additional assumption that the Ricci curvature
of $M$ is bounded from below by a constant, where the proof involves the consideration of the level sets of the minimal positive Green's function. 
For spin manifolds of arbitrary dimension, under topological assumptions of a different nature from ours, sharp estimates for the bottom spectrum are available 
from the work of Davaux \cite{D}, and more recently from \cite{WZ1, Liu}. For the rigidity issue, while we do not address it here, we refer to \cite{MW, WZ2, BC}
for relevant results.

Our second purpose is to prove a nonexistence result of positive Green's function for complete three dimensional manifolds with scalar curvature bounded from below
by a positive constant. Again, the same topological assumption as in the previous theorem is required in order to exclude the example alluded earlier. Indeed,
a complete manifold with positive bottom spectrum always admits a positive Green's function.

\begin{theorem}
\label{B} Let $\left( M,g\right) $ be a three dimensional complete
Riemannian manifold with scalar curvature $S\geq 1$ on $M.$ 
Suppose that $M$ has finitely many ends and its
first Betti number $b_{1}(M)<\infty.$ Then $M$ does not admit any positive Green's function.
\end{theorem}

While the assumption on the scalar curvature lower bound can certainly be relaxed, for example, to the form $S\geq \frac{c}{r(x)+1},$
where $r(x)$ is the geodesic distance to a fixed point, the optimal form is unclear to us presently. As pointed out above, the topological assumption is
necessary for the theorem to hold.

The proofs of both results rely on the $\mu$-bubble techniques. By definition, $\mu$-bubbles are surfaces that are stationary
for the prescribed mean curvature functional. Recall it was Schoen and Yau \cite{SY} who have initiated the program of using 
minimal surfaces to study three dimensional manifolds with nonnegative scalar curvature. $\mu$-bubbles, as natural extensions of minimal surfaces,
seem to be more versatile in studying scalar curvature. The usage of $\mu$-bubbles is not new. It has already appeared in the work of Gromov \cite{G1},
as well as in the work of Witten and Yau \cite{WY}, where they proved the nonexistence of wormholes for a class of conformally compact
Einstein manifolds. There has been much more development in recent years. In addition to \cite{G}, we refer to \cite{Z, CL, CLS, APX, HKKZ} 
and the references therein for further information.

For the proof of Theorem \ref{A}, we argue by contradiction and use the positive eigenfunction $w$ corresponding to the bottom spectrum
to construct a sequence of minimizing warped $\mu$-bubbles. The second variation formula and the largeness of the bottom spectrum enable
us to obtain a sequence of compact surfaces with uniform area bound and energy bound for the function $\ln w$ on the surfaces. This information
then leads to the conclusion that the volume of the whole manifold must be finite, an obvious contradiction to the positivity of the bottom spectrum.
The proof of Theorem \ref{B} follows a parallel strategy but with very different technical details. Again, one uses $\mu$-bubbles to construct a sequence
of compact surfaces with the squared integral of the mean curvature uniformly bounded. If the manifold admits a positive Green's function, then 
the existence of such surfaces allows one to conclude that the minimal positive Green's function must have finite $L^1$-energy. 
However, this is impossible by the co-area formula.

The paper is arranged as follows. In Section 2, we recall some preliminary facts concerning warped $\mu$-bubbles. Section 3 is devoted to the proof 
of the upper bound of the bottom spectrum, while Section 4 contains the proof of the nonexistence of positive Green's function.

\textbf{Acknowledgment:}   The first author was
partially supported by a Simons Foundation
grant.

\section{Warped $\mu$-bubbles}

In this section, we collect some basic facts about the warped $\mu$-bubbles which will be needed for our proofs.
The main references are \cite{Z}, \cite{G} and \cite{CL}. 

Consider an $n$-dimensional compact Riemannian manifold $N$ with boundary and assume its boundary $\partial N=\partial_{+} N\cup \partial_{-}N$, a union of two
disjoint components. For a smooth function $u>0$ on $N$ and a smooth function $h$ on the interior of $N$ with $h\to +\infty$ on $\partial_{+} N$ and $h\to -\infty$ 
on $\partial_{-} N,$ define the following functional 

\begin{equation}
P\left( \Omega \right) :=\int_{\partial ^{\ast }\Omega }u-\int_{N}\left(
\chi _{\Omega }-\chi _{\Omega _{0}}\right) h\,u \label{P}
\end{equation}
for all Caccioppoli subsets $\Omega\subset N$ with the symmetric difference
$\Omega\Delta \Omega_0$ contained in the interior of $N,$ where 
$\partial ^{\ast }\Omega$ is the reduced boundary of $\Omega$
given by $\partial \Omega\setminus \partial N$ and $\Omega_0$ is a fixed smooth domain in $N$
with $\partial_{+} N\subset \partial \Omega_0$ and $\partial_{-} N\cap \overline{\Omega_0}=\emptyset.$

When $n\leq 7,$ there exists a smooth domain $\Omega$ minimizing the functional $P$ (see \cite{Z, CL} for a proof). 
Such $\Omega$ is called a warped $\mu$-bubble following \cite{CL}. The existence proof shows that the reduced boundary
$\partial ^{\ast }\Omega$ of a warped $\mu$-bubble must be of positive distance away from both $\partial_{+} N$ and $\partial_{-} N.$

By the first variation formula,
the mean curvature $H$ of $\partial ^{\ast }\Omega$ is given by
\begin{equation}
H=h-\frac{u_{\nu }}{u},  \label{a2}
\end{equation}
where $\nu$ is the outward unit normal vector to $\Omega.$
Moreover, by the second variation formula, one has
\begin{eqnarray*} 
0 &\leq &\int_{\partial ^{\ast }\Omega }\left( \left\vert \nabla _{\partial ^{\ast }\Omega 
}\psi \right\vert ^{2}u-\left\vert A\right\vert ^{2}\psi ^{2}u-
\frac{u_{\nu }^{2}}{u}\psi ^{2}-\mathrm{Ric}(\nu, \nu)\,\psi
^{2}u\right) \\
&&+\int_{\partial ^{\ast }\Omega }\left( \left( \Delta u-\Delta _{\partial
^{\ast }\Omega }u\right) \psi ^{2}-H\,u_{\nu}\psi ^{2}-h_{\nu }\psi
^{2}u\right)
\end{eqnarray*}
for all $\psi \in C^{\infty }\left( \partial ^{\ast }\Omega \right),$ where 
$A$ and $\Delta _{\partial ^{\ast }\Omega }$ denote the second fundamental form
and the Laplacian of $\partial ^{\ast}\Omega,$ respectively.

In the case $n=3,$ by using a rearrangement idea due to Schoen and Yau \cite{SY}, 
the second variation formula can be rewritten into (see \cite{CL})
\begin{eqnarray}\label{a3}
0 &\leq &\int_{\partial ^{\ast }\Omega }\left( \left\vert \nabla _{\partial ^{\ast }\Omega 
}\psi \right\vert ^{2}u-\frac{1}{2}\left\vert A\right\vert ^{2}\psi ^{2}u-%
\frac{1}{2}\frac{u_{\nu }^{2}}{u}\psi ^{2}+K_{\partial ^{\ast }\Omega }\psi
^{2}u\right) \\
&&+\int_{\partial ^{\ast }\Omega }\left( \left( \Delta u-\Delta _{\partial
^{\ast }\Omega }u\right) \psi ^{2}-\frac{1}{2}h^{2}\psi ^{2}u-h_{\nu }\psi
^{2}u-\frac{1}{2}S\psi ^{2}u\right) \notag
\end{eqnarray}
for all $\psi \in C^{\infty }\left( \partial ^{\ast }\Omega \right),$ where $S$ is the scalar curvature of $M$ and
$K_{\partial^{\ast }\Omega }$ the Gauss curvature of $\partial ^{\ast}\Omega,$ respectively.

Now consider a complete noncompact three-dimensional manifold $M^{3}$ with one end and finite first Betti number. 
Let $d$ be a smoothing of the distance function to a fixed point $p\in M$ such that $d$ remains proper and satisfies 
$\left\vert \nabla d\right\vert \leq 2.$ 
For sufficiently large $L_0$ with all the representatives of the first homology $H_1(M)$ lying inside the set $\{d<L_0\},$
according to \cite{LT}, the boundary $\partial E_L$ of the only unbounded component $E_L$ of $M\setminus \{d\leq L\}$  
must be connected for all $L>L_0.$ 
In the following sections, we will consider warped $\mu$-bubbles in a $3$-dimensional manifold $N=E_{L_1}\setminus E_{L_2},$
where $L_1<L_2$ are regular values of $d$ with $L_0<L_1<L_2.$ Obviously,  
the set $N$ is a compact smooth manifold with boundary
components $\partial _{+}N=\partial E_{L_1}$ and $\partial _{-}N=\partial E_{L_2}.$
The function $h=h\left( d\right) $ will be taken as a smooth function depending on $d$ with
\begin{equation}
\lim_{d\rightarrow L_1}h\left( d\right) =+\infty ,\ \ \ \ \
\lim_{d\rightarrow L_2}h\left( d\right) =-\infty.  \label{a1}
\end{equation}

\begin{proposition}\label{X}
Let $M$ be a complete noncompact three dimensional manifold $M^{3}$ with one end and finite first Betti number
and $N=E_{L_1}\setminus E_{L_2}$ as above. Then for a warped $\mu$-bubble $\Omega$ in $N,$ there exists  
a connected component $\Sigma$ of $\partial ^{\ast }\Omega $ that separates the point
$p\in M$ from the infinity of $M.$ 
\end{proposition}

\begin{proof}
As $M$ has only one end, the set $E_{L_1}\setminus \Omega$ has one unbounded component $E_{\Omega}.$ 
Arguing as in \cite{LT}, one concludes that the smooth boundary of $E_{\Omega}$ must be connected. 
As noted above, the reduced boundary $\partial ^{\ast }\Omega $ is of positive distance from $\partial E_{L_1}.$
So $\partial E_{\Omega}$ must be a component of $\partial ^{\ast }\Omega.$ Obviously, it separates the point
$p\in M$ from the infinity of $M.$ 
This shows that $\Sigma=\partial E_{\Omega}$
is the desired component.
\end{proof}

\section{Bottom spectrum estimate}

We are now ready to prove the following sharp comparison result for the bottom spectrum. 

\begin{theorem}
Let $\left( M^{3},g\right) $ be complete noncompact three dimensional with
scalar curvature $S\geq -6K$ for $K\geq 0.$ Assume that $M$ has finitely
many ends and finite first Betti number. Then the bottom spectrum must satisfy
$\lambda _{1}\left( M\right)\leq K.$
\end{theorem}

\begin{proof}
Our argument is inspired by the work \cite{APX, HKKZ}.
Suppose by contradiction that there exists $\varepsilon >0$ such that 
\begin{equation}
\lambda _{1}\left( M\right) \geq K+\varepsilon.  \label{a0}
\end{equation}
Then there exists smooth function $w>0$ satisfying
\begin{equation}
\Delta w=-\lambda _{1}\left( M\right) w\text{ \ on \ }M.  \label{w}
\end{equation}

Since our argument uses only the geometry of $M$ outside an arbitrarily large compact
set, we may assume without loss of generality that $M$ has one end with
finite first Betti number.
As mentioned in the previous section, we will construct warped
$\mu$-bubbles in annulus region $N=E_{L_1}\setminus E_{L_2}$ of the end $E,$
where $$L_1=R-L \;\;\;\text{and}\;\;\; L_2=R+L,$$ with $R$ being an arbitrarily large number and $L$ a fixed number
to be chosen. Set $\Omega_0=E_{L_1}\setminus E_{R},$ $$u=w^{\gamma },$$ and 
\begin{equation*}
h\left( d\right) =\frac{2\pi }{\delta L}\tan \left( \frac{\pi }{2L}(R-d)\right)
\end{equation*} 
in (\ref{P}), where 
\begin{equation*}
\gamma =\frac{6}{2+\delta }
\end{equation*}
and 
\begin{equation*}
\delta =\frac{\varepsilon }{K+1}.
\end{equation*}
Without loss of generality, we may assume that $\varepsilon<\frac{1}{2}$, so  that $\delta <\frac{1}{2}$ as well. 
Obviously, the function $h=h\left( d\right) $ satisfies
\begin{equation*}
\lim_{d\rightarrow L_1}h\left( d\right) =+\infty ,\ \ \ \ \
\lim_{d\rightarrow L_2}h\left( d\right) =-\infty .  
\end{equation*}
So there exists a smooth warped $\mu$-bubble $\Omega.$ Moreover, by Proposition \ref{X},
there is a connected component $\Sigma \subset \partial ^{\ast }\Omega $
that separates $p\in M$ from the infinity of $M.$ On $\Sigma,$ the second variation formula (\ref{a3}) becomes
\begin{eqnarray*}
0 &\leq &\int_{\Sigma}\left( \left\vert \nabla _{\Sigma
}\psi \right\vert ^{2}u-\frac{1}{2}\left\vert A\right\vert ^{2}\psi ^{2}u-%
\frac{1}{2}\frac{u_{\nu }^{2}}{u}\psi ^{2}+K_{\Sigma}\psi
^{2}u\right) \\
&&+\int_{\Sigma }\left( \left( \Delta u-\Delta _{\Sigma}u\right) \psi ^{2}-\frac{1}{2}h^{2}\psi ^{2}u-h_{\nu }\psi
^{2}u-\frac{1}{2}S\psi ^{2}u\right)
\end{eqnarray*}%
for all $\psi \in C^{\infty }\left( \Sigma \right).$ 

For $\psi =\frac{1}{\sqrt{u}}$ it follows that%
\begin{eqnarray}
0 &\leq &\int_{\Sigma }\left( \frac{1}{4}\frac{\left\vert \nabla _{\Sigma
}u\right\vert ^{2}}{u^{2}}-\frac{1}{2}\frac{u_{\nu }^{2}}{u^{2}}+\frac{%
\Delta u-\Delta _{\Sigma }u}{u}\right)  \label{a3'} \\
&&+\int_{\Sigma }\left( -\frac{1}{2}\left\vert A\right\vert ^{2}+K_{\Sigma }-%
\frac{1}{2}h^{2}-h_{\nu }-\frac{1}{2}S\right) .  \notag
\end{eqnarray}%
Integrating by parts we have that 
\begin{equation*}
-\int_{\Sigma }\frac{\Delta _{\Sigma }u}{u}=\int_{\Sigma }\frac{\left\vert
\nabla _{\Sigma }u\right\vert ^{2}}{u^{2}}.
\end{equation*}%
Moreover, in view of (\ref{a2}), we have
\begin{equation*}
\left\vert A\right\vert ^{2}\geq \frac{1}{2}H^{2}=\frac{1}{2}\left( h-\frac{%
u_{\nu }}{u}\right) ^{2}.
\end{equation*}%
Therefore, (\ref{a3'}) becomes%
\begin{eqnarray}
0 &\leq &\int_{\Sigma }\left( -\frac{3}{4}\frac{\left\vert \nabla _{\Sigma
}u\right\vert ^{2}}{u^{2}}-\frac{1}{2}\frac{u_{\nu }^{2}}{u^{2}}-\frac{1}{4}%
\left( h-\frac{u_{\nu }}{u}\right) ^{2}\right)  \label{a4} \\
&&+\int_{\Sigma }\left( \frac{\Delta u}{u}+K_{\Sigma }-\frac{1}{2}%
h^{2}-h_{\nu }-\frac{1}{2}S\right) .  \notag
\end{eqnarray}%
Using the fact that $S\geq -6K$ together with the Gauss-Bonnet formula $\int_{\Sigma }K_{\Sigma }\leq 4\pi,$ we rewrite 
(\ref{a4}) into 
\begin{eqnarray}
0 &\leq &\int_{\Sigma }\left( -\frac{3}{4}\frac{\left\vert \nabla _{\Sigma
}u\right\vert ^{2}}{u^{2}}-\frac{3}{4}\frac{u_{\nu }^{2}}{u^{2}}+\frac{1}{2}h%
\frac{u_{\nu }}{u}\right)  \label{a5} \\
&&+\int_{\Sigma }\left( \frac{\Delta u}{u}-\frac{3}{4}h^{2}+\left\vert
\nabla h\right\vert +3K\right) +4\pi .  \notag
\end{eqnarray}%
As $\delta>0$, we have the inequality 
\begin{equation*}
\frac{1}{2}h\frac{u_{\nu }}{u}\leq \frac{3}{4}\frac{1}{1+\delta }h^{2}+\frac{%
1}{12}\left( 1+\delta \right) \frac{u_{\nu }^{2}}{u^{2}}.
\end{equation*}%
Plugging into (\ref{a5}) implies that
\begin{eqnarray*}
0 &\leq &\int_{\Sigma }\left( -\frac{3}{4}\frac{\left\vert \nabla _{\Sigma
}u\right\vert ^{2}}{u^{2}}-\frac{8-\delta }{12}\frac{u_{\nu }^{2}}{u^{2}}%
\right) \\
&&+\int_{\Sigma }\left( \frac{\Delta u}{u}-\frac{3}{4}\frac{\delta }{%
1+\delta }h^{2}+\left\vert \nabla h\right\vert +3K\right) +4\pi .
\end{eqnarray*}%
In particular, since $\delta<\frac{1}{2}$, 
this proves%
\begin{equation}
0\leq \int_{\Sigma }\left( -\frac{8-\delta }{12}\frac{\left\vert \nabla
u\right\vert ^{2}}{u^{2}}+\frac{\Delta u}{u}-\frac{1}{2}\delta
h^{2}+\left\vert \nabla h\right\vert +3K\right) +4\pi .  \label{a6}
\end{equation}%
Noting that
\begin{equation*}
h\left( d\right) =\frac{2\pi }{\delta L}\tan \left( \frac{\pi }{2L}(R-d)\right)
\end{equation*}%
and that $\left\vert \nabla d\right\vert \leq 2,$ one gets
\begin{equation*}
-\frac{1}{2}\delta h^{2}+\left\vert \nabla h\right\vert \leq \frac{2\pi ^{2}%
}{\delta L^{2}}.
\end{equation*}%
So (\ref{a6}) becomes
\begin{equation}
0\leq \int_{\Sigma }\left( -\frac{8-\delta }{12}\frac{\left\vert \nabla
u\right\vert ^{2}}{u^{2}}+\frac{\Delta u}{u}+3K+\frac{2\pi ^{2}}{\delta L^{2}%
}\right) +4\pi.  \label{a7}
\end{equation}
Since $u=w^{\gamma }$, we conclude from (\ref{a7}) that
\begin{equation}
0\leq \int_{\Sigma }\left( -\lambda _{1}\left( M\right) \gamma -\left(
\gamma -\frac{4+\delta }{12}\gamma ^{2}\right) \frac{\left\vert \nabla
w\right\vert ^{2}}{w^{2}}+3K+\frac{2\pi ^{2}}{\delta L^{2}}\right) +4\pi.
\label{a9}
\end{equation}%
Recalling that 
\begin{equation*}
\gamma =\frac{6}{2+\delta } \;\;\;\text{and}\;\;\;
\delta =\frac{\varepsilon }{K+1},
\end{equation*}
we have 
\begin{equation*}
-\left( \gamma -\frac{4+\delta }{12}\gamma ^{2}\right) =-\frac{3\delta }{%
\left( 2+\delta \right) ^{2}}\leq -\frac{\delta }{2},
\end{equation*}
and
\begin{equation*}
-\lambda _{1}\left( M\right) \gamma +3K\leq -\varepsilon.
\end{equation*}
Hence, (\ref{a9}) becomes 
\begin{equation*}
0\leq \int_{\Sigma }\left( -\frac{\varepsilon }{2\left( K+1\right) }\frac{%
\left\vert \nabla w\right\vert ^{2}}{w^{2}}-\varepsilon +\frac{2\left(
K+1\right) \pi ^{2}}{\varepsilon L^{2}}\right) +4\pi .
\end{equation*}%
By setting
\begin{equation*}
L=\frac{2\left( K+1\right) \pi }{\varepsilon }
\end{equation*}%
we obtain that 
\begin{equation*}
\int_{\Sigma }\left( \frac{\left\vert \nabla w\right\vert ^{2}}{w^{2}}%
+1\right) \leq \frac{8\left( K+1\right) \pi }{\varepsilon }.
\end{equation*}

In conclusion, for each large $R>1,$ there exists a compact surface $\Sigma_{R},$ 

\begin{equation*}
\Sigma _{R}\subset B_{p}\left( R+2L\right) \setminus B_{p}\left( R-2L\right),
\end{equation*}
that separates the fixed point $p\in M$ from the infinity of $M$ and satisfies

\begin{equation}
\int_{\Sigma _{R}}\left( \frac{\left\vert \nabla w\right\vert ^{2}}{w^{2}}%
+1\right) \leq \frac{8\left( K+1\right) \pi }{\varepsilon }.  \label{a10}
\end{equation}
To complete our argument, let $N_{R}$ be the unbounded component of $M\setminus \Sigma _{R}$ and 
$D_{R}:=M\setminus N_{R}$. That is, $D_{R}$ is the bounded domain of $M$ with
boundary $\Sigma _{R}.$ Clearly,  as $L$ is a fixed constant, there exists a sequence of $R_i\to \infty$
such that $D_{R_i}$ is increasing in $i$ and
\begin{equation}
M=\cup _{i=1}^\infty D_{R_i}.  \label{a11}
\end{equation}

By (\ref{w}), we have
\begin{equation*}
\lambda _{1}\left( M\right) =-\Delta \ln w-\left\vert \nabla \ln
w\right\vert ^{2}\leq -\Delta \ln w.
\end{equation*}%
Integrating the equation over $D_R,$ we conclude that 
\begin{eqnarray*}
\lambda _{1}\left( M\right) \mathrm{Vol}\left( D_{R}\right) &\leq
&-\int_{D_{R}}\Delta \ln w \\
&\leq &\int_{\Sigma _{R}}\left\vert \nabla \ln w\right\vert \\
&\leq &\frac{1}{2}\int_{\Sigma _{R}}\left( \left\vert \nabla \ln
w\right\vert ^{2}+1\right) \\
&\leq &\frac{4\left( K+1\right) \pi }{\varepsilon },
\end{eqnarray*}
where we have used (\ref{a10}).
Hence, the volume of $D_{R}$ is uniformly bounded from above independent of $R.$ In view of (\ref{a11}), it shows that
$M$ must have finite volume as well. This is an obvious contradiction to the fact that the bottom spectrum of $M$ is positive, see Chapter 22 in \cite{Li}.
\end{proof}

In particular, we have the following consequence.

\begin{corollary}
Let $\left( M^{3},g\right) $ be complete noncompact three dimensional with
scalar curvature $S\geq 0$. Assume that $M$ has finitely many ends and
finite first Betti number. Then $\lambda _{1}\left( M\right) =0$.
\end{corollary}

The following result strengthens the above corollary.

\begin{theorem}
Let $\left( M^{3},g\right) $ be a complete noncompact three dimensional manifold with
scalar curvature $S\geq 0.$ Assume that $M$ has finitely many ends and
finite first Betti number. Then there exists a sequence of geodesic balls $B_p(R_i)$ with
$R_{i}\rightarrow \infty $ such that their first Dirichlet eigenvalues satisfy
 
\begin{equation*}
\lambda _{1}\left( B_{p}\left( R_{i}\right) \right) \leq \frac{C}{R_{i}^{2}}
\end{equation*}%
for a universal constant $C>0.$
\end{theorem}

\begin{proof}
We argue by contradiction. Suppose that
\begin{equation}
\lambda _{1}\left( B_{p}\left( R\right) \right) \geq \frac{A}{R^{2}}
\label{m1}
\end{equation}
for all $R>R_{0},$ where both $R_0$ and $A>1$ are constants as large as one wants to specify.

Again, without loss of generality, we may assume that $M$ has one end and its first
Betti number is zero. As in the previous theorem, 
we will construct warped
$\mu$-bubbles in annulus region $N=E_{L_1}\setminus E_{L_2}$ of the end $E,$
where $L_1=R-L$ and $L_2=R+L$ with $R$ being an arbitrarily large number and $L=\frac{R}{4}.$ 
In (\ref{P}), set $\Omega_0=E_{L_1}\setminus E_{R},$ $h$ the function given by
\begin{equation*}
h\left( d\right) =\frac{2\pi }{L}\tan \left( \frac{\pi }{2L}(R-d)\right)
\end{equation*} 
and $u$ the positive Dirichlet eigenfunction of $B_{p}\left(2R\right),$ that is,
\begin{equation*}
\Delta u=-\lambda _{1}\left( B_{p}\left( 2R\right) \right) u.
\end{equation*}
Then there exists a smooth warped $\mu$-bubble $\Omega$ in $N.$ Moreover, by Proposition \ref{X},
there is a connected component $\Sigma \subset \partial ^{\ast }\Omega $
that separates $p\in M$ from the infinity of $M.$

Now the second variation formula (\ref{a4}) implies that
\begin{equation*}
0\leq \int_{\Sigma }\left( \frac{\Delta u}{u}-\frac{1}{2}\frac{\left\vert
\nabla u\right\vert ^{2}}{u^{2}}+K_{\Sigma }-\frac{1}{2}h^{2}-h_{\nu }-\frac{%
1}{2}S\right) .
\end{equation*}%
Since $S\geq 0$ and $\int_{\Sigma }K_{\Sigma }\leq 4\pi $, we get 
\begin{equation}
\int_{\Sigma }\left( -\frac{\Delta u}{u}+\frac{1}{2}\frac{\left\vert \nabla
u\right\vert ^{2}}{u^{2}}+\frac{1}{2}h^{2}-\left\vert \nabla h\right\vert
\right) \leq 4\pi .  \label{m2}
\end{equation}%
Noting that $\Sigma \subset B_{p}\left( 2R\right)\setminus B_{p}\left( \frac{R}{2}\right) $ and that
$h$ satisfies
\begin{equation*}
\frac{1}{2}h^{2}-\left\vert \nabla h\right\vert \geq -\frac{C}{R^{2}},
\end{equation*}%
we conclude from (\ref{m2}) that 
\begin{equation*}
\int_{\Sigma }\left( \lambda _{1}\left( B_{p}\left( 2R\right) \right) -\frac{%
C}{R^{2}}+\frac{1}{2}\frac{\left\vert \nabla u\right\vert ^{2}}{u^{2}}%
\right) \leq 4\pi .
\end{equation*}%
By (\ref{m1}) it follows that 
\begin{equation}
\int_{\Sigma }\frac{\left\vert \nabla u\right\vert }{u}\leq CR.  \label{m3}
\end{equation}%
Denoting with $D$ the compact domain in $M$ with $\partial D=\Sigma,$ we have
\begin{eqnarray*}
\lambda _{1}\left( B_{p}\left( 2R\right) \right) \mathrm{Vol}\left( D\right) 
&=&\int_{D}\left( -\Delta \log u-\left\vert \nabla \log u\right\vert
^{2}\right)  \\
&\leq &-\int_{D}\Delta \log u \\
&\leq &\int_{\Sigma }\frac{\left\vert \nabla u\right\vert }{u} \\
&\leq &CR.
\end{eqnarray*}%
This shows that 
\begin{equation}
\lambda _{1}\left( B_{p}\left( 2R\right) \right) V_{p}\left( \frac{R}{2}%
\right) \leq CR  \label{m4}
\end{equation}%
for all $R>R_{0},$ where $V_p(r)$ denotes the volume of the geodesic ball $B_p(r).$
Appealing to (\ref{m1}), we conclude that 
\begin{equation}
V_{p}\left( R\right) \leq \frac{C_{0}}{A}R^{3}  \label{m5}
\end{equation}%
for all $R>R_{0},$ where $C_{0}>0$ is a universal constant. Let $\phi$ be 
the cut-off function satisfying $\phi =1$ on $B_{p}\left( R\right),$ $\phi =0$
on $M\setminus B_{p}\left( 2R\right) $ and $\left\vert \nabla \phi
\right\vert \leq \frac{2}{R}.$ Then
\begin{eqnarray*}
\lambda _{1}\left( B_{p}\left( 2R\right) \right) V_{p}\left( R\right)  &\leq
&\lambda _{1}\left( B_{p}\left( 2R\right) \right) \int_{M}\phi ^{2} \\
&\leq &\int_{M}\left\vert \nabla \phi \right\vert ^{2} \\
&\leq &\frac{4}{R^{2}}V_{p}\left( 2R\right).
\end{eqnarray*}%
Using (\ref{m1}) and (\ref{m5}) we conclude that 
\begin{equation*}
\frac{A}{(2R)^{2}}V_{p}\left( R\right) \leq \lambda _{1}\left( B_{p}\left(
2R\right) \right) V_{p}\left( R\right) \leq \frac{4}{R^{2}}V_{p}\left(
2R\right) \leq \frac{4}{R^{2}}\frac{C_{0}}{A}\left( 2R\right) ^{3}.
\end{equation*}%
Therefore,
\begin{equation*}
V_{p}\left( R\right) \leq \frac{2^7 C_{0}}{A^{2}}R^{3}
\end{equation*}%
for all $R>R_{0}.$ Iterating this argument $m$ times, we get
that 
\begin{equation*}
V_{p}\left( R\right) \leq \frac{2^{7m}C_{0}}{A^{m}}R^{3}
\end{equation*}%
for all $R>R_{0}.$ Letting $m\rightarrow \infty $ we arrive at an obvious contradiction if $A>2^7$.
\end{proof}

\section{Nonexistence of positive Green's function}

In this section we show that, under a suitable topological assumption, a three-dimensional complete manifold
with uniformly positive scalar curvature must be parabolic, namely, it does not admit
any positive Green's function. We continue to adopt the same notations as before.

\begin{theorem}
Let $\left( M^{3},g\right) $ be a complete noncompact three dimensional with
scalar curvature $S\geq 1.$ Assume that $M$ has finitely many ends and
finite first Betti number. Then $M$ does not admit any positive Green's function. 
\end{theorem}

\begin{proof}
We may assume without loss of generality that $M$ has only one end.  Let $R_0>0$ be sufficiently large so that representatives of $H_1 (M)$ are
contained in $B_p(R_0)$. Let $M_0$ be the unbounded connected component of $M\setminus \overline{ B_p (R_0)}$ and  $$\Gamma:= \partial M_0$$ its boundary. By \cite{LT}, we know that $\Gamma$ is connected. Without loss of generality, we may also assume that $\Gamma$ is smooth.  

 As in the previous section, for each sufficiently large $R,$ 
consider $N=\left\{ L_1\leq d\leq L_2\right\},$ 
a smooth manifold with boundary components $\partial _{+}N=\left\{d=L_1\right\} $ and 
$\partial _{-}N=\left\{ d=L_2\right\},$ where $L_1=R-L$ and $L_2=R+L$ with $L\gg 1$ being 
a fixed constant. Let $u=1$ and  $h$ the function given by
\begin{equation*}
h\left( d\right) =\frac{4\pi }{L}\tan \left( \frac{\pi }{2L}(R-d)\right). 
\end{equation*}
Clearly,   $h=h\left( d\right) $ satisfies
\begin{equation*}
\lim_{d\rightarrow L_1}h\left( d\right) =+\infty ,\ \ \ \ \
\lim_{d\rightarrow L_2}h\left( d\right) =-\infty .
\end{equation*}

For 
\begin{equation*}
\Omega _{0}:=\left\{ L_1<d<R\right\}, 
\end{equation*}
the functional
\begin{equation*}
P\left( \Omega \right) :=\int_{\partial ^{\ast }\Omega }-\int_{N}\left( \chi
_{\Omega }-\chi _{\Omega _{0}}\right) h
\end{equation*}
admits a smooth minimizer $\Omega$ in $N$ with $\Omega\setminus \Omega_0$ contained in the 
interior of $N.$ Moreover, there exists a component $\Sigma$ of  $\partial ^{\ast }\Omega$
that separates point $p$ from the infinity of $M.$ 
On $\Sigma,$ its mean curvature
\begin{equation}
H=h.  \label{z0}
\end{equation}
The second variation formula (\ref{a3}) for $\psi =1$ together with the inequality 

\begin{equation*}
|A|^2\geq \frac{1}{2}H^2
\end{equation*}
gives

\begin{equation*}
0\leq \int_{\Sigma }\left( K_{\Sigma }-\frac{3}{4}h^{2}-h_{\nu }-\frac{1}{2}S\right) .
\end{equation*}%
Since $S\geq 1$ and $\int_{\Sigma }K_{\Sigma }\leq 4\pi,$ we conclude
\begin{equation}
\int_{\Sigma }\left( \frac{3}{4}h^{2}-\left\vert \nabla h\right\vert +\frac{1%
}{2}\right) \leq 4\pi .  \label{z1}
\end{equation}
As
\begin{equation*}
h\left( d\right) =\frac{4\pi }{L}\tan \left( \frac{\pi }{2L}(R-d)\right),
\end{equation*}
it follows that  
\begin{equation*}
\frac{1}{4}h^{2}-\left\vert \nabla h\right\vert \geq -\frac{4\pi ^{2}}{L^{2}}.
\end{equation*}%
Hence, for $L$ large enough, we get by (\ref{z1}) that  
\begin{equation*}
\int_{\Sigma }\left( h^{2}+1\right) \leq 16\pi.
\end{equation*}%
 Together with (\ref{z0}), it implies that 
\begin{equation}
\int_{\Sigma }H^{2}\leq 16\pi.  \label{z2}
\end{equation}

Let $D$ be the bounded domain of $M$ with boundary $\Sigma $ and $\Gamma$. 
Define $u>0$ to be the solution of the following Dirichlet problem 
\begin{eqnarray}
\Delta u &=&0\text{ \ \ in } D , \notag\\
u &=&1\text{ \ on } \Gamma , \label{D}\\
u &=&0\text{ \ on }\Sigma.\notag
\end{eqnarray}

We claim that 
\begin{equation}
\int_{D}\left( 1-\frac{u}{2}\right) \Delta \left\vert \nabla
u\right\vert \leq C_0+8\pi,  \label{z6}
\end{equation}
where $$
C_0:=\frac{1}{2} \int_{\Gamma }\left\vert \nabla
u\right\vert_\nu +\frac{1}{2} \int_{\Gamma }\left\vert \nabla
u\right\vert,
$$
and $\nu =\frac{\nabla u}{\left\vert \nabla u\right\vert }$ denotes the unit
normal to the level sets of $u$.

Indeed, the Stokes theorem implies that 
\begin{eqnarray*}
\int_{D}\left( \big(1-\frac{u}{2}\big)\Delta \left\vert \nabla u\right\vert -\left\vert
\nabla u\right\vert \Delta \big(1-\frac{u}{2}\big)\right) &=&\int_{\Gamma
}\left(  \big(1-\frac{u}{2}\big)\left\vert \nabla u\right\vert _{\nu }-\left\vert \nabla
u\right\vert \big(1-\frac{u}{2}\big)_\nu\right)\\
&&-\int_{\Sigma
}\left( \big(1-\frac{u}{2}\big)\left\vert \nabla u\right\vert _{\nu }-\left\vert \nabla
u\right\vert \big(1-\frac{u}{2}\big)_\nu\right).
\end{eqnarray*}
Since $u$ is a solution of (\ref{D}) and $\nu =\frac{\nabla u}{\left\vert \nabla u\right\vert }$, it follows that 
\begin{equation} \label{z3}
\int_{D} \big(1-\frac{u}{2}\big)\Delta \left\vert \nabla u\right\vert =C_0-\int_{\Sigma
}\left( \left\vert \nabla u\right\vert _{\nu }+\frac{1}{2}\left\vert \nabla
u\right\vert^2\right).
\end{equation}
As $\Delta u=0,$ the mean curvature of $\Sigma $ can be computed as 
\begin{equation*}
H=-\frac{1}{\left\vert \nabla u\right\vert }\left\vert \nabla u\right\vert
_{\nu }.
\end{equation*}%
Therefore, (\ref{z3}) implies that 
\begin{equation}  \label{z4}
\int_{D} \big(1-\frac{u}{2}\big)\Delta \left\vert \nabla u\right\vert \leq C_0+\int_{\Sigma
}\left\vert H\right\vert \left\vert \nabla u\right\vert-\frac{1}{2}\int_\Sigma \left\vert \nabla
u\right\vert^2.
\end{equation}
  By (\ref{z2}), 
\begin{eqnarray*}
\int_{\Sigma }\left\vert H\right\vert \left\vert \nabla u\right\vert  &\leq &%
\frac{1}{2}\int_{\Sigma }H^{2}+\frac{1}{2}\int_{\Sigma }\left\vert \nabla
u\right\vert ^{2} \\
&\leq &8\pi +\frac{1}{2}\int_{\Sigma }\left\vert \nabla u\right\vert ^{2}.
\end{eqnarray*}%
Combining with (\ref{z4}) we obtain that   
\begin{equation*}
\int_{D}\big(1-\frac{u}{2}\big)\Delta \left\vert \nabla u\right\vert \leq C_0+8\pi. 
\end{equation*}
This proves (\ref{z6}). 

Using the Bochner formula 
\begin{equation*}
\Delta \left\vert \nabla u\right\vert =\left( \left\vert u_{ij}\right\vert
^{2}-\left\vert \nabla \left\vert \nabla u\right\vert \right\vert
^{2}\right) \left\vert \nabla u\right\vert ^{-1}+\mathrm{Ric}\left( \nabla
u,\nabla u\right) \left\vert \nabla u\right\vert ^{-1}
\end{equation*}
together with the Gauss curvature equations on the surface $\{u=t\}$, see e.g. \cite{MW},
\begin{equation*}
\mathrm{Ric}\left( \nabla u,\nabla u\right) \left\vert \nabla u\right\vert
^{-1}=\left( \frac{1}{2}S-\frac{1}{2}S_{t }\right) \left\vert \nabla
u\right\vert +\left( \left\vert \nabla \left\vert \nabla u\right\vert
\right\vert ^{2}-\frac{1}{2}\left\vert \nabla ^{2}u\right\vert ^{2}\right)
\left\vert \nabla u\right\vert ^{-1},
\end{equation*}%
one obtains that
\begin{equation*}
\Delta \left\vert \nabla u\right\vert \geq \left( \frac{1}{2}S-\frac{1}{2}%
S_{t}\right) \left\vert \nabla u\right\vert,
\end{equation*}%
where $S_{t}$ is the scalar curvature of the surface $\{u=t\}$. Noting that 
$1-\frac{u}{2}>0$ and $S\geq 1,$ it follows that 
\begin{eqnarray}\label{z7}
\left( 1-\frac{u}{2}\right)\Delta \left\vert \nabla
u\right\vert  \geq \frac{1}{2}\left( 1-\frac{u}{2}\right)
 \left\vert \nabla u\right\vert- \frac{1}{2}\left( 1-\frac{u}{2}\right)
S_{t} \left\vert \nabla u\right\vert
\end{eqnarray}
on $\{u=t\}$. 
By the co-area formula we have
\begin{eqnarray*}
\int_{D}\left( 1-\frac{u}{2}\right) \Delta \left\vert \nabla
u\right\vert=\int_0^1 \left(1-\frac{t}{2}\right)\left(\int_{\{u=t\}} |\nabla u|^{-1}\Delta |\nabla u| \right)dt.
\end{eqnarray*}
Together with (\ref{z7}), this implies that 
\begin{equation*}
\int_D \left( 1-\frac{u}{2}\right)\Delta \left\vert \nabla
u\right\vert  \geq\frac{1}{2}\int_{D}\left( 1-\frac{u}{2}\right)
\left\vert \nabla u\right\vert-\frac{1}{2}\int_0^1 \left(1-\frac{t}{2}\right)\left(\int_{\{u=t\}}S_t\right) dt.
\end{equation*}
Since the level sets of $u$ are connected by \cite{MW}, the Gauss-Bonnet theorem yields 
\begin{equation*}
\frac{1}{2}\int_{0}^{1}\left( 1-\frac{t}{2}\right) \left(
\int_{\{u=t\}}S_{t }\right) dt 
\leq 4\pi \int_{0}^{1}\left( 1-\frac{t}{2}\right) dt
\leq 3\pi.
\end{equation*}
In conclusion,

\begin{eqnarray*}
\int_{D}\left( 1-\frac{u}{2}\right) \Delta \left\vert \nabla
u\right\vert  &\geq &\frac{1}{2}\int_{D}\left( 1-\frac{u}{2}\right)
\left\vert \nabla u\right\vert -3\pi \\
&\geq &\frac{1}{4}\int_{D}\left\vert \nabla u\right\vert -3\pi.
\end{eqnarray*}%
Together with (\ref{z6}), we conclude that
\begin{equation}
\int_{D}\left\vert \nabla u\right\vert \leq 4(11\pi +C_0).  \label{z8}
\end{equation}

Recall that such hypersurface $\Sigma$ and function $u$ are  constructed for each large $R.$
We label them as 
$\Sigma_{R}:=\Sigma \subset B_{p}\left( R+L\right) \setminus B_{p}\left( R-L\right)$ and $u_R$ to indicate
their dependency on $R$ now. By (\ref{z8}),
\begin{equation}\label{uR}
\int_{D_{R}}\left\vert \nabla u_{R}\right\vert \leq C_R,
\end{equation}
where 
$$
C_R:=44\pi +2 \int_{\Gamma }\left\vert \nabla
u_R\right\vert_\nu +2 \int_{\Gamma }\left\vert \nabla
u_R\right\vert,
$$ and $u_{R}$ solves the following Dirichlet problem  
\begin{eqnarray*}
\Delta u_{R} &=&0\text{ \ \ in } D_{R}, \\
u_{R} &=&1\text{ \ on }\Gamma , \\
u_{R} &=&0\text{ \ on }\Sigma _{R}.
\end{eqnarray*}%
Suppose by contradiction that $\left( M,g\right) $ admits a positive Green's function, i.e.,  it is nonparabolic. Then
the sequence $\left\{ u_{R}\right\} $ must have a subsequence that converges
uniformly on every compact subset of $M_0$ to a nonconstant harmonic function $w$ on 
$M_0$, see Chapter 17 in \cite{Li}. Here we recall that $M_0$ denotes the unbounded connected component of $M\setminus \overline{ B_p (R_0)}$, and $\Gamma=\partial M_0$.
 
Consequently, in view of (\ref{uR}),
\begin{equation}
\int_{K}\left\vert \nabla w\right\vert \leq C_1 \label{z9}
\end{equation}%
for any compact subset $K\subset M_0$, where 
$$
C_1:=44\pi +2 \int_{\Gamma }\left\vert \nabla
w\right\vert_\nu +2 \int_{\Gamma }\left\vert \nabla
w\right\vert.
$$
 However,  $w$ being a nonconstant harmonic function implies that 
\begin{equation*}
0<\int_{\Gamma }\left( -\frac{\partial w}{\partial
\nu }\right) =\int_{\partial B_{p}\left( r\right)\cap M_0 }\left( -\frac{\partial w}{%
\partial r}\right) \leq \int_{\partial B_{p}\left( r\right) \cap M_0}\left\vert
\nabla w\right\vert.
\end{equation*}%
So the co-area formula gives that 
\begin{equation*}
\int_{\left(B_{p}\left( R\right) \setminus B_{p}\left( R_0\right) \right) \cap M_0}\left\vert \nabla
w\right\vert =\int_{R_0}^{R}\left( \int_{\partial B_{p}\left( r\right)\cap M_0
}\left\vert \nabla w\right\vert \right) dr\geq C\, (R-R_0)
\end{equation*}
for a constant $C>0.$ Since this is true for all $R>R_0$ we have reached a
contradiction to (\ref{z9}). This means that $\left( M,g\right) $ is
parabolic. 
\end{proof}

\end{document}